\documentclass{amsart}
\usepackage{enumitem}
\usepackage{amsrefs}

\newtheorem{thm}{Theorem}[section]

\newtheorem{lem}[thm]{Lemma}
\newtheorem{prop}[thm]{Proposition}
\newtheorem{ex}[thm]{Example}
\numberwithin{equation}{section}

\newcommand{\<}{\begin{equation}}
\renewcommand{\>}{\end{equation}}
\let\oo=\infty

\author{Javad Rastegari}
\address{Department of Mathematics,
University of Western Ontario,
London, Canada}
\email{jrastega@uwo.ca}

\author{Gord Sinnamon}
\address{Department of Mathematics,
University of Western Ontario,
London, Canada}
\email{sinnamon@uwo.ca}
\thanks{Supported by the Natural Sciences and Engineering Research Council of Canada}

\begin{document}

\title{Weighted Fourier inequalities via rearrangements}

\begin{abstract} The method of using rearrangements to give sufficient conditions for Fourier inequalities between weighted Lebesgue spaces is revisited. New results in the case $q<p$ are established and a comparison between two known sufficient conditions is completed. In addition, examples are given to show that a simple weight condition that is sufficient for the weighted Fourier inequality in the cases $2<q<p$ and $1<q<p<2$ is no longer sufficient in the case $1<q<2<p$, contrary to statements in Theorems 1 and 4 of,``Weighted Fourier inequalities: new proofs and generalizations'', J. Fourier Anal. Appl. 9 (2003), 1--37. Several alternatives are given for strengthening the simple weight condition to ensure sufficiency in that case.
\end{abstract}

\maketitle

\section{Introduction}
Fix a positive integer $n$. For which indices $p$, $q$ and which weights $U$ and $W$ does there exist a constant $C<\oo$ such that the Fourier inequality
\<\label{Lebesgue}
\bigg(\int_{\mathbb R^n} |\hat f(y)|^qU(y)\,dy\bigg)^{1/q}
\le C\bigg(\int_{\mathbb R^n}|f(x)|^pW(x)\,dx\bigg)^{1/p}
\>
holds for all $f\in L^1(\mathbb R^n)$? If there is such a constant then the Fourier transform is a bounded operator from a dense subspace of $L^p_W(\mathbb R^n)$ into $L^q_U(\mathbb R^n)$. Here the Fourier transform of a function in $L^1(\mathbb R^n)$ is defined as an integral operator by
\[
\hat f(x)=\int_{\mathbb R^n}e^{-2\pi ix\cdot y}f(y)\,dy.
\]
Using the inequality above, if $C$ is finite, a standard argument will extend the integral operator to a linear operator on the whole space $L^p_W(\mathbb R^n)$. So for our purposes it will suffice to restrict attention to functions in $L^1(\mathbb R^n)$.

In the case $1<p\le q<\oo$, the question was addressed by three nearly simultaneous but different approaches appearing in 1983-84, all involving rearrangements of the weights. See \cites{BH,H1,JS1, JS2,M1,M2}. The case $1< p<\oo$, $0<q<p$ has been considered as well, in \cites{H2,BHJ} and later in \cite{BH2003}, but clarification and improvement of this case is still possible. This is the object of the present paper.

The non-increasing rearrangement of $f^*$ of a $\mu$-measurable function $f$ is defined for $t>0$ by,
\[
f^*(t)=\inf\{\alpha>0:\mu_f(\alpha)\le t\},\quad\text{where}\quad
\mu_f(\alpha)=\mu\{x:|f(x)|>\alpha\}.
\]
We refer to \cite{BS} for standard properties of the non-increasing rearrangement, just mentioning  two: {\it Hardy's Lemma} states that if $f_1$ and $f_2$ are non-negative measurable functions on $(0,\oo)$ then 
\[
\int_0^\oo f_1g\le\int_0^\oo f_2g
\]
holds for all non-negative non-decreasing $g$ if and only if it holds with $g=\chi_{(t,\oo)}$ for each $t>0$. The {\it Hardy-Littlewood-Polya Inequality} shows that for any non-negative measurable $f,g$ defined on $\mathbb R^n$, 
\[
\int_{\mathbb R^n}fg\le\int_0^\oo f^*g^*.
\]

Given non-negative, Lebesgue measurable functions $U$ and $W$ on $\mathbb R^n$, define $u$ and $w$ by $u=U^*$ and $1/w=(1/W)^*$. The functions $u$ and $w$ are defined on $(0,\oo)$, they take values in $[0,\oo]$, $u$ is non-increasing and $w$ is non-decreasing. The rearranged Fourier inequality below, expressed in terms of the weights $u$ and $w$, gives a sufficient condition for (\ref{Lebesgue}). For a proof see, for example, the proof of \cite{BH2003}*{Theorem 1}.

\begin{prop}\label{Lorentz implies Lebesgue} If $p,q\in (0,\oo)$ and
\<\label{Lorentz}
\bigg(\int_0^\oo (\hat f)^*(t)^qu(t)\,dt\bigg)^{1/q}
\le C\bigg(\int_0^\oo f^*(s)^pw(s)\,ds\bigg)^{1/p}, \quad f\in L^1(\mathbb R^n),
\>
then (\ref{Lebesgue}) holds with the same constant $C$.
\end{prop}

Working with (\ref{Lorentz}) instead of (\ref{Lebesgue}) is the essence of the ``rearrangement''  approach to Fourier inequalities in weighted Lebesgue spaces. It has proven to be a powerful method but it is not the only approach. See, for example, \cites{SW, S3}.

The strategy we adopt for proving inequality (\ref{Lebesgue}) via inequality (\ref{Lorentz}) with monotone $u$ and $w$ begins by using the mapping properties of the Fourier transform to find weighted Hardy-type inequalities that imply (\ref{Lorentz}). This is done in Section \ref{R2H}. Then, known weight characterizations for the Hardy-type inequalities are employed to find sufficient conditions on the weights and indices for (\ref{Lorentz}) and hence for (\ref{Lebesgue}) to hold. This is carried out in Section \ref{EWC}.

The mapping properties we use here are universally known: The Fourier transform maps $L^1$ to $L^\oo$ and maps $L^2$ to $L^2$. In \cite{JT},  Jodiet and Torchinsky showed that for any operator with these mapping properties there exists a constant $D$ such that
\<\label{JT46}
\int_0^x(\hat f)^*(t)^2\,dt\le D\int_0^x\bigg(\int_0^{1/t}f^*\bigg)^2\,dt
\>
for all $f\in L^1+L^2$. 

The main Hardy inequality we use is the same one employed, either explicitly or implicitly, in the various 1983-4 papers \cites{BH,H1,JS1, JS2,M1,M2}:
\<
\label{Main Hardy}
\bigg(\int_0^\oo\bigg(\int_0^{1/t} f\bigg)^qu(t)\,dt\bigg)^{1/q}
\le C_1\bigg(\int_0^\oo f^pw\bigg)^{1/p},\qquad f\ge0.
\>
A characterization of weights for this inequality is easily derived from known results; details are in Proposition \ref{Main Hardy Weights} below. The case $p\le q$ illustrates the simple sufficient conditions that this method gives for the rearranged Fourier transform inequality (\ref{Lorentz}). The following result was (essentially) given in each of \cites{BH, BH2003, H1, JS2, M2}.
\begin{prop} \label{BH0} Let $u=U^*$ and $1/w=(1/W)^*$, and suppose $1<p\le q<\oo$. If
\<\label{BH0cond}
\sup_{x>0}\bigg(\int_0^{1/x}u\bigg)^{1/q}\bigg(\int_0^x w^{1-p'}\bigg)^{1/p'}<\oo
\>
then (\ref{Lorentz}) holds, and hence (\ref{Lebesgue}) holds. 
\end{prop}
Throughout, conjugate indices are denoted with a prime, so $1/p+1/p'=1$. See \cite{BH2003} for estimates of the best constant $C$ in (\ref{Lebesgue}), depending only on the indices $p$, $q$, the constant $D$ from (\ref{JT46}), and the value of the above supremum. 

In Section \ref{Comp14} we complete a comparison begun in \cite{BH2003} of two theorems proved there. Both give sufficient conditions for (\ref{Lebesgue}), but the form of the conditions differ. We show that one of the two includes the results of the other.

When we are using (\ref{Lorentz}) only to get to (\ref{Lebesgue}) we may assume that $u$ is non-increasing and $w$ is non-decreasing. Indeed, the monotonicity of the weights will figure prominently in our analysis. But the inequality (\ref{Lorentz}) is of interest in its own right, without the monotonicity restrictions on $u$ and $w$. It expresses the boundedness of the Fourier transform between weighted Lorentz $\Lambda$-spaces and has been studied in \cite{BH2003}. Work on the corresponding inequality for Lorentz $\Gamma$-spaces in \cites{S1,S2} resulted in necessary and sufficient conditions and was extended to the Fourier coefficient map in \cites{R, RS}. We introduce the Lorentz $\Lambda$- and $\Gamma$-spaces, and improve (weaken) the known sufficient conditions for (\ref{Lorentz}), using $\Gamma$-space techniques, in Section \ref{Lambda}. 

Section \ref{Other} looks at extending the results of the previous sections to Fourier transforms of functions on groups other than $\mathbb R^n$.

An example is given in Section \ref{counterexamples} to show that the expected sufficient condition for (\ref{Lebesgue}), known to be valid when $2<q<p$ and $1<q<p<2$, fails to be sufficient when $1<q<2<p$. 

Although, in (\ref{Lebesgue}), we consider all positive values of the index $q$, we restrict our attention to $p>1$. In fact, the case $p=1$ of (\ref{Lebesgue}) is rarely discussed, because of the following simple argument. With $C=(1/\operatorname{ess\,inf} W)(\int_{\mathbb R^n}U)^q$, the trivial estimate $|\hat f|\le\int_{\mathbb R^n}|f|$ yields (\ref{Lebesgue}). On the other hand, for $L^1$ functions $f$ approaching a point mass at $x$ (in a suitable weak sense) $|\hat f|$ approaches a constant function. Applying (\ref{Lebesgue}) to such $f$ yields $(\int_{\mathbb R^n}U)^q\le C w(x)$. Taking the essential infimum over $x$ now shows that $C=(1/\operatorname{ess\,inf} W)(\int_{\mathbb R^n}U)^q$ is best possible in (\ref{Lebesgue}). 

A variation of this argument shows that if $0<p<1$ then (\ref{Lebesgue}) holds only if either $U$ is almost everywhere zero or $W$ is almost everywhere infinite. 

\section{Reduction to Hardy-type inequalities}\label{R2H}

The success of the rearrangement approach to (\ref{Lebesgue}) in the case $1<p\le q<\oo$ is well known. But methods and expectations change when $q<p$, because
arguments based on the Fourier transform being of type $(p,q)=(1,\oo)$ and $(p,q)=(2,2)$ lend themselves most naturally to the case $p\le q$. Most authors were content to consider only that case. However, in \cites{H2, BHJ} a general estimate due to Calder\'on, based on weak-type mapping properties instead of the strong-type mapping properties mentioned above, was successfully used to give sufficient conditions for the Fourier inequality (\ref{Lebesgue}) when $1<q<p$. 

Later, in \cite{BH2003}, the strong-type conditions, combined with a duality argument, provided improved sufficient conditions using a reduction to the Hardy inequality (\ref{Main Hardy}). However,  the argument that provided these improved conditions is only applicable under the restriction $\max(p',q)\ge2$. This restriction is used implicitly in the proofs but, unfortunately, was not included in the statements of \cite{BH2003}*{Theorems 1 and 4}. In Section \ref{counterexamples} we demonstrate that the restriction is essential by giving an example to show that the statements of \cite{BH2003}*{Theorems 1 and 4} may fail in the case $\max(p',q)<2$. 

In our first main result, we give simple a priori conditions on indices and weights that enable the reduction to (\ref{Main Hardy}) go through even when the restriction $\max(p',q)\ge2$ does not hold. 
Note that parts \ref{1} and \ref{2} are contained in the proof of \cite{BH2003}*{Theorem 1}. Parts \ref{5} and \ref{8} were suggested by \cite{JS1}*{Corollary 3}. For parts \ref{4} and \ref{7} we need to define the weight condition $B_p$ for $p>0$: A non-negative function $f$ is in $B_p$ provided there exists a constant $\beta<\oo$
\[
y^p\int_y^\oo f(x)\,\frac{dx}{x^p}\le\beta\int_0^y f(x)\,dx,\quad y>0.
\]

\begin{thm} \label{Thm1} Let $u=U^*$ and $1/w=(1/W)^*$. Suppose $1<p<\oo$ and $0<q<p$, and suppose that the inequality (\ref{Main Hardy}) holds for some constant $C_1$. If any one of the following conditions holds, then so does the Fourier inequality (\ref{Lebesgue}).
\begin{enumerate}[label=\rm{(\alph*)}]
\item\label{1} $q\ge2$;
\item\label{2} $q>1$ and $p\le2$;
\item\label{3} there exists a constant $\beta$ such that for all $y>0$;
\[
y^{q/2}\int_y^\oo u(x)\,\frac{dx}{x^{q/2}}
\le \beta\bigg(\int_0^y u(x)\,dx+y^{\max(1,q)}\int_y^\oo u(x)\,\frac{dx}{x^{\max(1,q)}}\bigg);
\]
\item\label{4} $u\in B_{q/2}$;
\item\label{5} $q>1$ and $t^{2-q}u(t)$ is a decreasing function of $t$;
\item\label{6} $q>1$ and there exists a constant $\beta$ such that for all $y>0$,
\[
y^{p'/2}\int_y^\oo w(x)^{1-p'}\,\frac{dx}{x^{p'/2}}
\le \beta\bigg(\int_0^y w(x)^{1-p'}\,dx+y^{p'}\int_y^\oo w(x)^{1-p'}\,\frac{dx}{x^{p'}}\bigg);
\]
\item\label{7} $q>1$ and $w^{1-p'}\in B_{p'/2}$;
\item\label{8} $q>1$ and $t^{2-p'}w(t)^{1-p'}$ is a decreasing function of $t$.
\end{enumerate}
\end{thm}

Explicit estimates of the constant $C$ in (\ref{Lebesgue}) may be given in terms of the indices $p$ and $q$, and the constants $D$, $\beta$, and $C_1$. We omit the details.

If none of these a priori conditions holds, then some additional conditions are needed to ensure the validity of (\ref{Lebesgue}). We approach the problem by adding another Hardy-type inequality to (\ref{Main Hardy}) so that together the two imply (\ref{Lorentz}). This has already been done, in \cites{H2, BHJ}, but it was based on weaker mapping properties of the Fourier transform than we use here and, predictably, gives a more restrictive sufficient condition than we obtain using the strong mapping properties. It is included, as (\ref{BHJ Hardy}), because it leads to a more tractable weight condition.

\begin{thm}\label{Thm2} Let $u=U^*$ and $1/w=(1/W)^*$. Suppose $1<p<\oo$ and $0<q<p$, and suppose that the inequality (\ref{Main Hardy}) holds. If either of the following inequalities also holds, for all $f\in L^1(\mathbb R^n)$, then so does the rearranged Fourier inequality (\ref{Lorentz}) and hence also the Fourier inequality (\ref{Lebesgue}):
\<
\label{BHJ Hardy}
\bigg(\int_0^\oo\bigg(x^{-1/2}\int_{1/x}^\oo t^{-1/2}f^*(t)\,dt\bigg)^qu(x)\,dx\bigg)^{1/q}
\le C_2\bigg(\int_0^\oo (f^*)^pw\bigg)^{1/p}.
\>
\<
\label{New Hardy}
\bigg(\int_0^\oo\bigg(\frac1x\int_{1/x}^\oo (f^*)^2\bigg)^{q/2}u(x)\,dx\bigg)^{1/q}
\le C_3\bigg(\int_0^\oo (f^*)^pw\bigg)^{1/p}.
\>
\end{thm}

Once again, we omit the details of the available estimates for the constant $C$ in (\ref{Lebesgue}) in terms of the indices $p$ and $q$ and the constants $D$, $C_1$, and $C_2$ or $C_3$.

Before proceeding to the proofs of these two theorems a discussion of duality is needed. We show that, when both $p>1$ and $q>1$, each of the inequalities (\ref{Lebesgue}) and (\ref{Main Hardy}) holds if and only if its counterpart, obtained by the replacements $p\mapsto q'$, $q\mapsto p'$, $U\mapsto W^{1-p'}$ and $W\mapsto U^{1-q'}$, also holds. Observe that positive exponents commute with the rearrangement, so
\[
u=U^*\mapsto(W^{1-p'})^*=((1/W)^*)^{p'-1}=(1/w)^{p'-1}=w^{1-p'}
\]
and
\[
1/w=(1/W)^*\mapsto(1/U^{1-q'})^*=(U^*)^{q'-1}=1/u^{1-q'}.
\] 
Thus, the replacements above imply $u\mapsto w^{1-p'}$ and $w\mapsto u^{1-q'}$.

\begin{lem}\label{LD} If $1<q<\oo$ and $1<p<\oo$ then for any weights $U$ and $W$, and any constant $C$, (\ref{Lebesgue}) holds for all $f\in L^1(\mathbb R^n)$ if and only if 
\<\label{Lebesgue Dual}
\bigg(\int_{\mathbb R^n}|\hat g|^{p'}W^{1-p'}\bigg)^{1/p'}\le C\bigg(\int_{\mathbb R^n}|g|^{q'}U^{1-q'}\bigg)^{1/q'}
\>
for all $g\in L^1(\mathbb R^n)$.
\end{lem}
\begin{proof} Suppose $(\ref{Lebesgue})$ holds with constant $C$ and fix $g\in L^1(\mathbb R^n)$. For any $f\in L^1(\mathbb R^n)$ with $\int_{\mathbb R^n} |f|^pW\le 1$, we have $(\int_{\mathbb R^n}|\hat f|^q U)^{1/q}\le C$ so H\"older's inequality yields,
\[
\bigg|\int_{\mathbb R^n}f\hat g\bigg|
=\bigg|\int_{\mathbb R^n}\hat fg\bigg|
\le C\bigg(\int_{\mathbb R^n}|g|^{q'}U^{1-q'}\bigg)^{1/q'}.
\]
Taking the supremum over all such $f$, and using the density of $L^1(\mathbb R^n)$ in $L^p_W(\mathbb R^n)$, gives (\ref{Lebesgue Dual}). The reverse implication is proved similarly.\end{proof}
\begin{lem}\label{MHD} If $1<q<\oo$ and $1<p<\oo$ then for any weights $u$ and $w$, and any constant $C_1$, (\ref{Main Hardy}) holds for all non-negative measurable $f$ on $(0,\oo)$ if and only if 
\<\label{Main Hardy Dual}
\bigg(\int_0^\oo\bigg(\int_0^{1/x} g\bigg)^{p'}w(x)^{1-p'}\,dx\bigg)^{1/p'}
\le C_1\bigg(\int_0^\oo g^{q'}u^{1-q'}\bigg)^{1/q'}
\>
for all non-negative measurable $g$ on $(0,\oo)$.
\end{lem}
\begin{proof} 
Suppose $(\ref{Main Hardy})$ holds with constant $C_1$ and fix a measurable $g\ge0$ on $(0,\oo)$. For any $f\ge0$ with $\int_0^\oo f^pw\le 1$, we have $(\int_0^\oo(\int_0^{1/t} f)^qu(t)\,dt)^{1/q}\le C_1$ so H\"older's inequality yields,
\[
\int_0^\oo f(x)\int_0^{1/x}g(t)\,dt\,dx=\int_0^\oo\int_0^{1/t}f(x)\,dxg(t)\,dt
\le C_1\bigg(\int_0^\oo g^{q'}u^{1-q'}\bigg)^{1/q'}.
\]
Taking the supremum over all such $f$ gives (\ref{Main Hardy Dual}). The reverse implication is proved similarly.\end{proof}

It is worth pointing out that (\ref{Main Hardy}) requires the inequality to hold for all non-negative functions but, as we shall soon see, we will only apply it to non-increasing functions. Nothing is lost, however, for suppose (\ref{Main Hardy}) were known to hold for non-increasing functions. Then, for any $f$, 
\begin{align*}
\bigg(\int_0^\oo\bigg(\int_0^{1/t} |f|\bigg)^qu(t)\,dt\bigg)^{1/q}
&\le \bigg(\int_0^\oo\bigg(\int_0^{1/t} f^*\bigg)^qu(t)\,dt\bigg)^{1/q}\\
&\le C\bigg(\int_0^\oo (f^*)^pw\bigg)^{1/p}\\
&\le C\bigg(\int_0^\oo |f|^pw\bigg)^{1/p}.
\end{align*}
The last inequality is an exercise, using the Hardy-Littlewood-Polya inequality and the fact that $w$ is non-decreasing. Thus (\ref{Main Hardy}) holds for all functions.

In the next lemma we isolate an estimate that will be used in the proof of both main theorems.
\begin{lem} \label{CompoundHardyLemma} Let $0<q<\oo$ and $1<p<\oo$. If
\<\label{CompoundHardy}
\bigg(\int_0^\oo\bigg(\frac1x\int_0^x\bigg(\int_0^{1/t}f^*\bigg)^2
\,dt\bigg)^{q/2}u(x)\,dx\bigg)^{1/q}\le C_4\bigg(\int_0^\oo (f^*)^pw\bigg)^{1/p}
\>
for all $f\in L^1(\mathbb R^n)$ then (\ref{Lorentz}) holds with $C=D^{1/2}C_4$. Here $D$ is the constant of (\ref{JT46}).
\end{lem}
\begin{proof} For each $f\in L^1(\mathbb R^n)$, $(\hat f^*)^2$ is decreasing, so by (\ref{JT46}),
\[
(\hat f)^*(x)^2\le\frac1x\int_0^x(\hat f^*)^2
\le \frac Dx\int_0^x\bigg(\int_0^{1/t}f^*\bigg)^2\,dt.
\]
Thus,
\[
\bigg(\int_0^\oo (\hat f^*)(x)^qu(x)\,dx\bigg)^{1/q}
\le D^{1/2}\bigg(\int_0^\oo \bigg(\frac1x\int_0^x
\bigg(\int_0^{1/t}f^*\bigg)^2\,dt\bigg)^{q/2}u(x)\,dx\bigg)^{1/q}.
\]
The hypothesis completes the proof.
\end{proof}

We now turn to the proof of Theorem \ref{Thm1}.
\begin{proof} To see part \ref{1}, let $q\ge2$ and apply \cite{JT}*{Theorem 4.7} to get
\[
\int_0^\oo (\hat f^*)^qu\le D^{q/2}\int_0^\oo \bigg(\int_0^{1/t}f^*\bigg)^qu(t)\,dt,
\]
where $D$ is the constant from (\ref{JT46}). Since (\ref{Main Hardy}) is assumed to hold, we also have (\ref{Lorentz}) and hence (\ref{Lebesgue}).

For Part \ref2, the indices $p$ and $q$ are greater than $1$ so the assumption (\ref{Main Hardy}) and Lemma \ref{MHD} implies that (\ref{Main Hardy Dual}) also holds. But since $1<q<p\le 2$, we have $2\le p'<q'$ so part \ref{1} may be applied to conclude that (\ref{Lebesgue Dual}) holds. Lemma \ref{LD} shows that (\ref{Lebesgue}) holds, as required. 

Similar arguments show that, \ref{6}, \ref{7}, and \ref{8} follow from parts \ref{3}, \ref{4}, and \ref{5}, respectively. 

In view of part \ref{1} it suffices to establish \ref{3}, \ref{4}, and \ref{5} in the case $q<2$. A trivial estimate shows that part \ref{4} follows from part \ref{3}. To see that part \ref{5} follows from part \ref{4}, suppose $q>1$ and $t^{2-q}u(t)$ is decreasing. In this case,
\[
y^{q/2}\int_y^\oo u(x)\,\frac{dx}{x^{q/2}}
\le y^{q/2}y^{2-q}u(y)\int_y^\oo x^{q/2-2}\,dx
=\frac{yu(y)}{1-q/2}
\]
and
\[
\int_0^y u(x)\,dx\ge y^{2-q}u(y)\int_0^yx^{q-2}\,dx=\frac{yu(y)}{q-1}.
\]
It follows that $u\in B_{q/2}$. 

It remains to establish \ref{3} when $q<2$. By Lemma \ref{CompoundHardyLemma} it is enough to show that the condition of part \ref{3} implies inequality 
(\ref{CompoundHardy}). Let $\alpha=\max(1,q)$. Since $f^*$ is non-increasing and right continuous, it is an increasing pointwise limit of continuous non-increasing functions. (For instance, convolve $f$ with $(1/n)\chi_{(1,e^{1/n})}$ in $((0,\oo),dt/t)$ for $n=1,2,\dots$.) Therefore, we may assume without loss of generality that $f^*$ is continuous. For such an $f$, let $F(t)=-(\int_0^{1/t}f^*)^q$. We show that $t^{\alpha+1}F'(t)$ is non-decreasing in $t$ by considering two cases. When $0<q\le1$, $\int_0^{1/t}f^*$ decreases with $t$ so
\[
t^{1+\alpha}F'(t)=q\bigg(\int_0^{1/t}f^*\bigg)^{q-1}f^*(1/t)
\]
is a non-decreasing function of $t$. When $1<q<2$, $t\int_0^{1/t}f^*$ increases with $t$ so
\[
t^{1+\alpha}F'(t)=q\bigg(t\int_0^{1/t}f^*\bigg)^{q-1}f^*(1/t)
\]
is a non-decreasing function of $t$. Using the hypothesis of part\ref{3}, we get
\begin{align*}
\int_y^\oo t^{q/2}\int_t^\oo u(x)\,\frac{dx}{x^{q/2}}\,\frac{dt}{t^{1+\alpha}}
&=\int_y^\oo\int_y^xt^{q/2}\,\frac{dt}{t^{1+\alpha}}u(x)\,\frac{dx}{x^{q/2}}\\
&\le\frac{y^{q/2-\alpha}}{\alpha-q/2}\int_y^\oo u(x)\,\frac{dx}{x^{q/2}}\\
&\le \frac{\beta y^{-\alpha}}{\alpha-q/2}\bigg(\int_0^yu(x)\,dx+y^\alpha\int_y^\oo u(x)\,\frac{dx}{x^{\alpha}}\bigg)\\
&=\frac{\beta\alpha}{\alpha-q/2}\int_y^\oo\int_0^tu(x)\,dx\,\frac{dt}{t^{1+\alpha}}.
\end{align*}
Since this holds for all $y$, and $t^{1+\alpha}F'(t)$ is non-decreasing, Hardy's lemma implies that
\[
\int_0^\oo t^{1+\alpha}F'(t)t^{q/2}\int_t^\oo u(x)\,\frac{dx}{x^{q/2}}\,\frac{dt}{t^{1+\alpha}}
\le\frac{\beta\alpha}{\alpha-q/2}\int_0^\oo t^{1+\alpha}F'(t)\int_0^tu(x)\,dx\,\frac{dt}{t^{1+\alpha}}.
\]
This simplifies to,
\[
\int_0^\oo\int_0^xt^{q/2}F'(t)\,dt u(x)\,\frac{dx}{x^{q/2}}
\le\frac{\beta\alpha}{\alpha-q/2}\int_0^\oo\int_x^\oo F'(t)\,dt\,u(x)\,dx.
\]
Now we apply Minkowski's integral inequality with index $2/q$ to get,
\begin{align*}
&\int_0^\oo\bigg(\frac1x\int_0^x\bigg(\int_0^{1/s}f^*\bigg)^2\,ds\bigg)^{q/2}u(x)\,dx\\
&=\int_0^\oo\bigg(\frac1x\int_0^x\bigg(\int_s^\oo F'(t)\,dt\bigg)^{2/q}\,ds\bigg)^{q/2} u(x)\,dx\\
&\le\int_0^\oo\int_0^\oo\bigg(\frac1x\int_0^{\min(x,t)}ds\bigg)^{q/2}F'(t)\,dtu(x)\,dx\\
&=\int_0^\oo\int_0^xt^{q/2}F'(t)\,dtu(x)\,\frac{dx}{x^{q/2}}
+\int_0^\oo\int_x^\oo F'(t)\,dtu(x)\,dx\\
&\le\bigg(1+\frac{\beta\alpha}{\alpha-q/2}\bigg)\int_0^\oo\int_x^\oo F'(t)\,dtu(x)\,dx\\
&=\bigg(1+\frac{\beta\alpha}{\alpha-q/2}\bigg)\int_0^\oo\bigg(\int_0^{1/x}f^*\bigg)^qu(x)\,dx.
\end{align*}
This reduces the proof of (\ref{CompoundHardy}) to our assumption that (\ref{Main Hardy}) 
holds.\end{proof}

Proof of Theorem \ref{Thm2}. 
\begin{proof} The first part of the theorem is contained in the proof of \cite{BHJ}*{Theorem 1.1(ii)}, but it may also be deduced from the second part as follows. By Minkowski's integral inequality,
\begin{align*}
\bigg(\frac1x\int_{1/x}^\oo f^*(s)^2\,ds\bigg)^{1/2}
&\le\bigg(\frac1x\int_{1/x}^\oo \bigg(\frac1s\int_0^sf^*(t)\,dt\bigg)^2\,ds\bigg)^{1/2}\\
&\le\int_0^\oo\bigg(\frac1x\int_{\max(1/x,t)}^\oo\,\frac{ds}{s^2}\bigg)^{1/2}f^*(t)\,dt\\
&=\int_0^{1/x} f^*(t)\,dt
+x^{-1/2}\int_{1/x}^\oo t^{-1/2}f^*(t)\,dt.
\end{align*}
This estimate and the extended Minkowski inequality show that (\ref{Main Hardy}) and (\ref{BHJ Hardy}) together imply (\ref{New Hardy}).

To prove the second part we will apply Lemma \ref{CompoundHardyLemma}. To begin, break the inner integral at $t=1/x$ and use the triangle inequality in the $L^2$ norm to get,
\begin{align*}
&\bigg(\frac1x\int_0^x\bigg(\int_0^{1/t}f^*\bigg)^2\,dt\bigg)^{1/2}\\
&\le\bigg(\frac1x\int_0^x\bigg(\int_0^{1/x}f^*\bigg)^2\,dt\bigg)^{1/2}
+\bigg(\frac1x\int_0^x\bigg(\int_{1/x}^{1/t}f^*\bigg)^2\,dt\bigg)^{1/2}\\
&=\int_0^{1/x}f^*
+\bigg(\frac1x\int_{1/x}^\oo\bigg(\frac1t\int_{1/x}^tf^*\bigg)^2\,dt\bigg)^{1/2}\\
&\le\int_0^{1/x}f^*
+\bigg(\frac1x\int_{1/x}^\oo\bigg(\frac1{t-1/x}\int_{1/x}^tf^*\bigg)^2\,dt\bigg)^{1/2}\\
&\le\int_0^{1/x}f^*
+2\bigg(\frac1x\int_{1/x}^\oo f^*(t)^2\,dt\bigg)^{1/2}.
\end{align*}
The last inequality above is the Hardy inequality on the interval $(1/x,\oo)$. Taking $c=\max(2^{1/q-1},1)$, and using the (extended) Minkowski inequality gives
\begin{align*}
&\bigg(\int_0^\oo\bigg(\frac1x\int_0^x \bigg(\int_0^{1/t}f^*\bigg)^2\,dt\bigg)^{q/2}u(x)\,dx\bigg)^{1/q}\\
&\le c\bigg(\int_0^\oo\bigg(\int_0^{1/x} f^*\bigg)^qu(x)\,dx\bigg)^{1/q}
+c\bigg(\int_0^\oo\bigg(\frac1x\int_{1/x}^\oo (f^*)^2\bigg)^{q/2}u(x)\,dx\bigg)^{1/q}.
\end{align*}
These estimates show that (\ref{CompoundHardy}) holds whenever both (\ref{Main Hardy}) and (\ref{New Hardy}) do. Now Lemma \ref{CompoundHardyLemma} completes the proof.
\end{proof}

In the above proof we showed that inequalities (\ref{Main Hardy}) and (\ref{New Hardy}) imply (\ref{CompoundHardy}). To see that nothing is lost by this decomposition, we observe that the other implication also holds. Since the square of $\int_0^{1/x} f^*(t)\,dt$ is a decreasing function of $x$,
\[
\int_0^{1/x} f^*(t)\,dt\le\bigg(\frac1x\int_0^x\bigg(\int_0^{1/t}f^*\bigg)^2\,dt\bigg)^{1/2}.
\]
Thus, if (\ref{CompoundHardy}) holds then (\ref{Main Hardy}) holds for decreasing functions. It follows from the remark after Lemma \ref{MHD}, that (\ref{Main Hardy}) holds for all non-negative functions.
Inequality (\ref{New Hardy}) also follows from (\ref{CompoundHardy}): Since $f^*$ is decreasing,
\[
\frac1x\int_{1/x}^\oo f^*(t)^2\,dt
\le\frac1x\int_{1/x}^\oo \bigg(\frac1t\int_0^tf^*\bigg)^2\,dt
=\frac1x\int_0^x\bigg(\int_0^{1/t}f^*\bigg)^2\,dt.
\]

\section{Explicit weight conditions}\label{EWC}

In the previous section, our approach was to find Hardy-type inequalities, depending on the weights $u$ and $w$ that imply the Fourier inequality (\ref{Lebesgue}) for the weights $U$ and $W$. This puts us in a position to use known weight characterizations for Hardy-type operators to give conditions that ensure the validity of (\ref{Lebesgue}). In this section we do exactly that, beginning with the inequality (\ref{Main Hardy}). 

\begin{prop}\label{Main Hardy Weights} Suppose $0<q<\oo$, $1< p<\oo$ and $1/r=1/q-1/p$. The inequality (\ref{Main Hardy}), that is,
\[
\bigg(\int_0^\oo\bigg(\int_0^{1/x}f\bigg)^q u(x)\,dx\bigg)^{1/q}
\le C_1\bigg(\int_0^\oo f^pw\bigg)^{1/p},\quad f\ge0,
\]
holds (for some finite constant $C_1$) if and only if:
\begin{enumerate}[label=\rm{(\alph*)}] 
\item $1<p\le q$ and
\[
\sup_{x>0}\bigg(\int_0^{1/x} u\bigg)^{1/q}\bigg(\int_0^x w^{1-p'}\bigg)^{1/p'}<\oo;\quad\text{or}
\]
\item\label{0<q<p} $0<q<p$, $1<p$, and
\<\label{MH}
\bigg(\int_0^\oo \bigg(\int_0^x u\bigg)^{r/p}\bigg( \int_0^{1/x}w^{1-p'}\bigg)^{r/p'}u(x)\,dx\bigg)^{1/r}<\oo.
\>
\end{enumerate}
\end{prop}
\begin{proof} The change of variable $x\mapsto1/x$ converts (\ref{Main Hardy}) to the standard form of the Hardy inequality found in \cite{B}*{Theorem 1} and \cite{SS}*{Theorem 2.4}. The same change of variable is used to convert the weight conditions given in those results to the ones above.
\end{proof}
It is pointed out in \cite{SS}*{page 93} that if $q>1$ or if $w^{1-p'}$ is locally integrable, then the condition in Case \ref{0<q<p} may be replaced by,
\<\label{alt}
\bigg(\int_0^\oo \bigg(\int_0^{1/x} u\bigg)^{r/q}\bigg( \int_0^xw^{1-p'}\bigg)^{r/q'}w(x)^{1-p'}\,dx\bigg)^{1/r}<\oo.
\>
The equivalence of these two forms (essentially using integration by parts) may be needed to reconcile previous results with those given here.

The next result applies to inequality (\ref{BHJ Hardy}) because (\ref{BHJ Hardy}) is obtained from (\ref{BHJ Hardy whole}) by taking $f=f^*$. Since (\ref{BHJ Hardy}) only requires that (\ref{BHJ Hardy whole}) hold for non-increasing functions, the weight condition (\ref{BHJH}) is sufficient for (\ref{BHJ Hardy}) but may be stronger than necessary. The condition (\ref{BHJH}) may be compared, via integration by parts, to \cite{BHJ}*{Condition (1.9)} when $q>1$. 

\begin{prop}\label{BHJ Hardy Weights} Suppose $0<q<p$, $1< p<\oo$ and $1/r=1/q-1/p$. The inequality
\<\label{BHJ Hardy whole}
\bigg(\int_0^\oo\bigg(x^{-1/2}\int_{1/x}^\oo t^{-1/2}f(t)\,dt\bigg)^q u(x)\,dx\bigg)^{1/q}
\le C_5\bigg(\int_0^\oo f^pw\bigg)^{1/p},\quad f\ge0,
\>
holds (for some finite constant $C_5$) if and only if
\<\label{BHJH}
\bigg(\int_0^\oo \bigg(\int_x^\oo t^{-q/2} u(t)\,dt\bigg)^{r/p}\bigg(\int_{1/x}^\oo t^{-p'/2}w(t)^{1-p'}\,dt\bigg)^{r/p'}x^{-q/2}u(x)\,dx\bigg)^{1/r}<\oo.
\>
\end{prop}
\begin{proof} This time we make the substitution $t\mapsto1/t$ on both sides of the inequality and replace $f(1/t)$ by $t^{3/2}g(t)$. This puts it in the standard form of \cite{SS}*{Theorem 2.4}, but with weights $x^{-q/2} u(x)$ and $t^{(3p-4)/2}w(1/t)$. As before, the same substitution is used to re-write the weight condition in the above form.
\end{proof}

Necessary and sufficient conditions on weights for which (\ref{New Hardy}) holds are known, but are substantially more complicated than for the other two inequalities. 
\begin{prop} \label{New Hardy Weights}
 Suppose $0<q<p$, $1<p<\oo$ and $1/r=1/q-1/p$. The inequality (\ref{New Hardy}), that is, 
\[
\bigg(\int_0^\oo\bigg(\frac1x\int_{1/x}^\oo (f^*)^2\bigg)^{q/2} u(x)\,dx\bigg)^{1/q}
\le C_3\bigg(\int_0^\oo (f^*)^pw\bigg)^{1/p},\quad f\in L^1(\mathbb R^n),
\]
holds (for some finite constant $C_3$) whenever:
\begin{enumerate}[label=\rm{(\alph*)}]
\item $0<q<p\le 2$ and
\<\label{seq}
\sup_{x_k}\bigg(\sum_{k\in\mathbb Z}\bigg(\int_{1/x_{k+1}}^{1/x_k}(x_{k+1}t-1)^{q/2}u(t)\,\frac{dt}{t^q}\bigg)^{r/q}\bigg(\int_0^{x_{k+1}}w\bigg)^{-r/p}\bigg)^{1/r}<\oo
\>
where the supremum is taken over all increasing sequences $x_k$, $k\in\mathbb Z$; or
\item $0<q<2<p$, (\ref{seq}), and
\begin{multline}\label{NH}
\bigg(\int_0^\oo\bigg(\int_{1/x}^\oo\bigg(\frac1{t-1/x}\int_0^tw\bigg)^{-p/(p-2)}w(t)\,dt\bigg)^{(r/2)(p-2)/p}\\
\times\bigg(\int_x^\oo t^{-q/2}u(t)\,dt\bigg)^{r/p}x^{-q/2}u(x)
\,dx\bigg)^{1/r}<\oo.
\end{multline}
\end{enumerate}
\end{prop}
\begin{proof} Every non-negative, decreasing function on $(0,\oo)$ can be represented as a limit of functions $(f^*)^2$ for $f\in L^1(\mathbb R^n)$. So, letting $x\mapsto1/x$ puts (\ref{New Hardy}) in the form of \cite{GS}*{Theorem 5.1}, with indices $p/2$ and $q/2$. Cases (ii) and (vi) of that theorem yield the results above, after letting $x\mapsto1/x$ again.
\end{proof}

We conclude this section with a summary of sufficient conditions for the Fourier inequality (\ref{Lebesgue}) in the case $q<p$.
\begin{thm}\label{Summary} Suppose $0<q<p$ and $1<p<\oo$. Let $U$ and $W$ be non-negative, measurable function on $\mathbb R^n$ and set $u=U^*$ and $1/w=(1/W)^*$. Inequality (\ref{Lebesgue}), that is,
\<
\bigg(\int_{\mathbb R^n} |\hat f(y)|^qU(y)\,dy\bigg)^{1/q}
\le C\bigg(\int_{\mathbb R^n}|f(x)|^pW(x)\,dx\bigg)^{1/p},\quad f\in L^1(\mathbb R^n),
\>
holds (for some finite constant $C$) provided 
(\ref{MH}) holds and \ref{BH1}, \ref{q>1} or \ref{q<1} is satisfied:  
\begin{enumerate}[label=\rm{(\roman*)}]
  \item\label{BH1} $2\le q< p$ or $1<q<p\le2$.
  \item\label{q>1} $1<q<2<p$ and
  \begin{enumerate}
    \item one or more of \ref{3}--\ref{8} from Theorem \ref{Thm1},
    \item  (\ref{BHJH}), or
    \item (\ref{seq}) and (\ref{NH}).
  \end{enumerate}
  \item\label{q<1}  $0<q<1<p$ and 
  \begin{enumerate}
    \item \ref{3} or \ref{4} from Theorem \ref{Thm1},
    \item (\ref{BHJH}),
    \item $p\le2$ and (\ref{seq}), or
    \item $p>2$ and (\ref{seq}) and (\ref{NH}).
  \end{enumerate}
\end{enumerate}
\end{thm}

Necessary and sufficient conditions for weighted Hardy inequalities may be expressed in a wide variety of different, but equivalent, forms. See, for example, \cites{OS1,OS2,W}. The form of the weight conditions in Propositions \ref{Main Hardy Weights}, \ref{BHJ Hardy Weights}, and \ref{New Hardy Weights} represent one choice but Theorems \ref{Thm1} and \ref{Thm2} may be combined with any of the various forms to give sufficient conditions for the Fourier inequality (\ref{Lebesgue}).

\section{Comparison of Theorems 1 and 4 in \cite{BH2003}}\label{Comp14}

In \cite{BH2003}*{Theorem 4}, under the a priori assumption that $w\in B_p$ or $u^{1-q'}\in B_{q'}$, sufficient conditions for (\ref{Lebesgue}) are given that appear different than those of \cite{BH2003}*{Theorem 1}. The latter appear here in Proposition \ref{BH0}, see condition (\ref{BH0cond}), and part \ref{BH1} of Theorem \ref{Summary}, see the equivalent conditions (\ref{MH}) and (\ref{alt}). The weight conditions from the two theorems are compared in \cite{BH2003}*{Remark 5} but neither theorem is shown to directly imply the other. The next lemma completes the comparison, showing that the two theorems give equivalent weight conditions whenever \cite{BH2003}*{Theorem 4} applies. We conclude that  \cite{BH2003}*{Theorem 1} is the stronger result because it does not require any a priori assumption.

\begin{lem}\label{1v4} If $1<p<\oo$, $w\in B_p$ and $w$ is increasing, then there is a constant $C_6$ such that
\<\label{equiv}
x\bigg(\int_0^x w\bigg)^{-1/p}\le \bigg(\int_0^x w^{1-p'}\bigg)^{1/p'}\le C_6x\bigg(\int_0^x w\bigg)^{-1/p}
\>
for all $x>0$.
\end{lem}
\begin{proof} By H\"older's inequality,
\[
x=\int_0^x w^{1/p}w^{-1/p}\le\bigg(\int_0^x w\bigg)^{1/p}\bigg(\int_0^xw^{1-p'}\bigg)^{1/p'}.
\]
This proves the first inequality of (\ref{equiv}). For the second, recall \cite{AM}*{Theorem 1}, which shows that since $w\in B_p$ there exists a constant $C_6<\oo$ such that the Hardy inequality
\[
\bigg(\int_0^\oo\bigg(\frac1t\int_0^tf\bigg)^pw(t)\,dt\bigg)^{1/p}
\le C_6\bigg(\int_0^\oo f^pw\bigg)^{1/p}
\]
holds for all non-increasing functions $f\ge0$. Fix $x>0$ and, for $k>1/x$, let $f_k=\max(w(1/k),w)^{1-p'}\chi_{(0,x)}$. Since $f_k$ is non-increasing, so is its moving average. Therefore, 
\[
\bigg(\int_0^x w\bigg)^{1/p}\frac1x\int_0^x f_k
\le\bigg(\int_0^\oo\bigg(\frac1t\int_0^t f_k\bigg)^p w(t)\,dt\bigg)^{1/p}
\le C_6\bigg(\int_0^x f_k^pw\bigg)^{1/p}.
\]
Since $f_k^pw\le f_k$, and $f_k$ is bounded above, we have
\[
\bigg(\int_0^x f_k\bigg)^{1/p'}\le C_6x\bigg(\int_0^x w\bigg)^{-1/p}.
\]
Letting $k\to\oo$ gives the second inequality of (\ref{equiv}). 
\end{proof}

Using this lemma we show, in four cases, that if the appropriate a priori condition holds, then the Fourier inequalities that follow from \cite{BH2003}*{Theorem 4}, are the same as those that follow from \cite{BH2003}*{Theorem 1}.

Case 1. Suppose $1<p\le q$, $q\ge2$, and $w\in B_p$. Then the weight pair $(u,w)$ satisfies the weight condition of \cite{BH2003}*{Theorem 4(i)} if and only if it satisfies the weight condition of \cite{BH2003}*{Theorem 1(i)}. That is,
\[
\sup_{x>0}x\bigg(\int_0^{1/x}u\bigg)^{1/q}\bigg(\int_0^x w\bigg)^{-1/p}<\oo
\]
if and only if (\ref{BH0cond}) holds. This follows directly from Lemma \ref{1v4}.

Case 2. Suppose $2\le q<p$, $1/r=1/q-1/p$, and $w\in B_p$. Then the weight pair $(u,w)$ satisfies the weight condition of \cite{BH2003}*{Theorem 4(ii)} if and only if it satisfies the weight condition of \cite{BH2003}*{Theorem 1(ii)}. That is,
\<\label{4ii}
\bigg(\int_0^\oo \bigg(\int_0^xu\bigg)^{r/p}\bigg(\int_0^{1/x} w\bigg)^{-r/p}x^{-r}u(x)\,dx\bigg)^{1/r}<\oo
\>
if and only if (\ref{alt}) holds. Lemma \ref{1v4} gives the equivalence of (\ref{4ii}) and (\ref{MH}); the remark following Proposition \ref{Main Hardy Weights} completes the argument.


Case 3. Suppose $1<p\le q<2$, and $u^{1-q'}\in B_{q'}$. Then the weight pair $(u,w)$ satisfies the weight condition of \cite{BH2003}*{Theorem 4(iii)} if and only if it satisfies the weight condition of \cite{BH2003}*{Theorem 1(i)}. That is,
\[
\sup_{x>0}\frac1x\bigg(\int_0^{1/x}u^{1-q'}\bigg)^{-1/q'}\bigg(\int_0^x w^{1-p'}\bigg)^{1/p'}<\oo
\]
if and only if (\ref{BH0cond}) holds. This follows from Lemma \ref{1v4}, with $p$ and $w$ replaced by $q'$ and $u^{1-q'}$.


The last case does not include the index range $1<q<2<p$, which should have been excluded in the statements of \cite{BH2003}*{Theorems 1 and 4}. We also fix a typographic error in the weight condition of \cite{BH2003}*{Theorem 4(iv)}. 

Case 4. Suppose $1<q<p\le2$, $1/r=1/q-1/p$, and $u^{1-q'}\in B_{q'}$. Then the weight pair $(u,w)$ satisfies the weight condition of \cite{BH2003}*{Theorem 4(iv)} if and only if it satisfies the weight condition of \cite{BH2003}*{Theorem 1(ii)}. That is,
\[
\bigg(\int_0^\oo \bigg(\int_0^xw^{1-p'}\bigg)^{r/q'}\bigg(\int_0^{1/x} u^{1-q'}\bigg)^{-r/q'}x^{-r}w(x)^{1-p'}\,dx\bigg)^{1/r}<\oo.
\]
if and only if (\ref{alt}) holds. This follows from Lemma \ref{1v4}, with $p$ and $w$ replaced by $q'$ and $u^{1-q'}$.


These four cases together show that \cite{BH2003}*{Theorem 1} is stronger than \cite{BH2003}*{Theorem 4} because it applies without the a priori conditions of \cite{BH2003}*{Theorem 4}.

\section{Lorentz space Fourier inequalities}\label{Lambda}

In this section we return briefly to inequality (\ref{Lorentz}), but without the monotonicity restrictions on the weights $u$ and $w$. Let $0<p<\oo$ and $w$ be a weight, and define the Lorentz $\Lambda$- and $\Gamma$-``norms'' of $f$ by,
\[
\|f\|_{\Lambda_p(w)}=\bigg(\int_0^\oo (f^*)^pw\bigg)^{1/p}\quad\text{and}\quad
\|f\|_{\Gamma_p(w)}=\bigg(\int_0^\oo \bigg(\frac1t\int_0^tf^*\bigg)^pw(t)\,dt\bigg)^{1/p}.
\]
Since $f^*$ is non-increasing,  $\|f\|_{\Lambda_p(w)}\le\|f\|_{\Gamma_p(w)}$ for any weight $w$. If $w\in B_p$, then \cite{AM}*{Theorem 1} shows that the two are equivalent.

With this notation, inequality (\ref{Lorentz}) becomes 
\<\label{LL}
\|\hat f\|_{\Lambda_q(u)}\le C\|f\|_{\Lambda_p(w)}.
\>
This inequality, which expresses the boundedness of the Fourier transform between Lorentz $\Lambda$-spaces, was considered in \cite{BH2003}*{Theorems 2 and 3}. Restricting the domain to the smaller $\Gamma$-space leads to the inequality 
\<\label{GL}
\|\hat f\|_{\Lambda_q(u)}\le C\|f\|_{\Gamma_p(w)}.
\>
These were studied in \cite{S1}, where it was shown in \cite{S1}*{Theorem 3.4} that for $0<p\le q<\oo$, inequality (\ref{GL}) holds whenever
\<\label{level}
\sup_{x<y}\bigg(\frac xy\int_0^y u\bigg)^{1/q}\bigg(x^p\int_0^{1/x}w(t)\,dt+\int_{1/x}^\oo w(t)\,\frac{dt}{t^p}\bigg)^{-1/p}<\oo.
\>

Observe that this result includes \cite{BH2003}*{Theorem 2(i)}; for if $w\in B_p$, inequalities (\ref{LL}) and (\ref{GL}) are equivalent. Also, if $u$ is non-increasing and
\[
\sup_{x>0}\frac1x\bigg(\int_0^x u\bigg)^{1/q}\bigg(\int_0^{1/x}w\bigg)^{-1/p}<\oo,
\]
then straightforward estimates show that (\ref{level}) also holds.

But, as we see next, results for the $\Gamma$-space inequality (\ref{GL}) may be used to further weaken sufficient conditions for the $\Lambda$-space inequality (\ref{LL}). This result may be viewed as a generalization of Proposition \ref{BH0}, without monotonicity conditions on the weights. We will need the {\it level function} $u^o$ of $u$ defined by requiring that the function $x\mapsto\int_0^x u^o$ is the least concave majorant of $x\mapsto\int_0^x u$. Note that $u^o$ is non-increasing and it coincides with $u$ when $u$ is non-increasing. For properties of the level function, see \cite{MS} and the references therein.

\begin{thm} Let $1<p\leq q <\infty$ and $q\geq 2$. Assume $u$ and $w$ are weight functions on $(0,\infty)$. If 
\<\label{levelcond}
\sup_{x>0} \bigg(\int_0^{1/x} u^o \bigg)^{1/q} \bigg( \int_0^x w^{1-p'} \bigg)^{1/p'} < \infty ,
\>
then there exists $C>0$ such that (\ref{LL}) holds for all $f\in L^1(\mathbb{R}^n)$.
\end{thm}
\begin{proof} If $\int_0^{1/x} u^o$ is infinite for some $x>0$, concavity shows that it is infinite for all $x>0$. But then the finiteness of (\ref{levelcond}) implies $w$ is infinite almost everywhere so the inequality (\ref{LL}) holds trivially. Henceforth, we assume that $\int_0^{1/x} u^o (t)\, dt<\infty$ for $x>0$. 

Since the concave function $t \mapsto \int_0^t u^o$ is absolutely continuous, we may set
\[
\sigma (t) = t^{q-2} u^o(1/t)  = - t^q \, \dfrac{d}{dt} \bigg(\int_0^{1/t} u^o\bigg)
\]
to get
\<\label{sigma}
\int_x^{\infty} \dfrac{\sigma (t)}{t^q} \, dt =\int_0^{1/x} u^o.
\>
Thus, for any $x<y$,
\[
\frac 1y\int_0^y u\le\frac 1y\int_0^y u^o\le\frac1x\int_0^x u^o
=\frac1x\int_{1/x}^{\infty} \dfrac{\sigma (t)}{t^q} \, dt.
\]
It follows that
\[
\sup_{x<y}\bigg(\frac xy\int_0^y u\bigg)^{1/q}\bigg(x^q\int_0^{1/x}\sigma(t)\,dt+\int_{1/x}^\oo \sigma(t)\,\frac{dt}{t^q}\bigg)^{-1/q}\le 1<\oo.
\]
So, as we have seen above, \cite{S1}*{Theorem 3.4} shows there is a constant $C_7$ such that
\[
\|\hat f\|_{\Lambda_q(u)}\le C_7\|f\|_{\Gamma_q(\sigma)}.
\]

Combining (\ref{sigma}) with the hypothesis (\ref{levelcond}) shows that
\[
\sup_{x>0} \bigg(\int_x^{\infty} \dfrac{\sigma (t)}{t^q} \, dt\bigg)^{1/q} \bigg( \int_0^x w(t)^{1-p'}\,dt \bigg)^{1/p'} < \infty,
\]
which, by \cite{B}*{Theorem 1}, implies there exists a constant $C_8$ such that
\[
\|f\|_{\Gamma_q(\sigma)}\le C_8\|f\|_{\Lambda_p(w)}.
\]
Taking $C=C_7C_8$ completes the proof.
\end{proof}
\section{Other Fourier transforms}\label{Other}

Suppose $(X,\lambda)$ and $(Y,\nu)$ are $\sigma$-finite measure spaces and $T$ is a linear map defined on $L^1(X)+L^2(X)$ taking values in $L^\oo(X)+L^2(Y)$. Further suppose that $T: L^1(X)\to L^\oo(Y)$ and $T: L^2(X)\to L^2(Y)$ are bounded maps. Then $T$ has a uniquely defined dual map $T'$ such that $T': L^1(Y)\to L^\oo(X)$ and $T': L^2(Y)\to L^2(X)$ are bounded maps.

The results of the previous section apply with the operator $T$ in place of the Fourier transform, as only the boundedness properties of the Fourier transform were used in an essential way. (Lemma \ref{LD} used the self-duality of the Fourier transform, but it is easy to see that the boundedness properties of $T'$ will suffice.) Here we understand that $U$ is a weight on $Y$ and $u$ is the rearrangement of $U$ with respect to the measure $\nu$. Also, $W$ is a weight on $X$ and $1/w$ is the rearrangement of $1/W$ with respect to the measure $\lambda$.

In particular, Propositions \ref{Lorentz implies Lebesgue} and \ref{BH0}, and Theorems \ref{Thm1}, \ref{Thm2}, and \ref{Summary} remain valid for the Fourier transform taken over any locally compact abelian group.

Depending on the underlying measures $\lambda$ and $\nu$, further simplification of the sufficient weight conditions may be possible. This is because the range of the rearrangement may not include all decreasing functions. For example, rearranging a function on a space of finite measure gives a decreasing function supported in a finite subinterval of $[0,\oo)$. Another example is a sequence, viewed as a function over a space with counting measure; its rearrangement is a decreasing function that is constant on $[k,k+1)$ for $k=0,1,2,\dots$. (Naturally, this function may be identified with the decreasing sequence of its values.) 

In the case of general measure spaces more complicated restrictions on the range of the rearrangement are possible. But for Haar measure on locally compact abelian groups these two examples are the only ones possible. (However, the two may combine; in the case of the finite Fourier transform both $\lambda$ and $\nu$ are counting measure on a finite set.)

To illustrate the kinds of simplifications that may be expected, we consider the specific case of the Fourier transform on $\mathbb T^n$  In this case the measure $\lambda$ may be identified with Lebesgue measure on $[0,1]^n$ and the measure $\nu$ is counting measure on $\mathbb Z^n$. The weight $U$ is defined on $\mathbb Z^n$ and therefore $u=U^*$ is constant on $[k,k+1)$ for $k=0,1,2,\dots$. The weight $W$ is supported on a set of measure $1$ so $w$, defined by $1/w=(1/W)^*$, is an increasing function on $[0,\oo)$ that takes the value $\oo$ on $[1,\oo)$. 

With these weights, the condition (\ref{MH}) is sufficient, in the appropriate range of Lebesgue indices, to imply the inequality corresponding to (\ref{Lebesgue}) for the Fourier transform on $\mathbb T^n$, namely
\[
\bigg(\sum_{y\in\mathbb Z^n} |\hat f(y)|^qU(y)\bigg)^{1/q}
\le C\bigg(\int_{\mathbb T^n}|f(x)|^pW(x)\,dx\bigg)^{1/p}.
\]
But $w^{1-p'}$ vanishes on $[1,\oo)$ and $u$ is a step function taking values $u(0), u(1),\dots$; essentially a sequence. It is natural to replace (\ref{MH}) by an equivalent condition in a form that recognizes these facts. One choice is,
\<\label{FS}
\bigg(u(0)^{r/q}\bigg(\int_0^1 w^{1-p'}\bigg)^{r/p'}+\sum_{k=1}^\oo\bigg(\sum_{j=0}^{k-1}u(j)\bigg)^{r/p}\bigg(\int_0^{1/k} w^{1-p'}\bigg)^{r/p'}u(k)\bigg)^{1/r}<\oo.
\>
To verify the equivalence, first observe that 
\[
\int_0^1\bigg(\int_0^xu\bigg)^{r/p}\bigg(\int_0^{1/x} w^{1-p'}\bigg)^{r/p'}u(x)\,dx=\frac qr u(0)^{r/p}\bigg(\int_0^1 w^{1-p'}\bigg)^{r/p'}u(0).
\]
Next, note that  $\frac1x\int_0^xu$ is non-increasing and $x\int_0^{1/x}w^{1-p'}$ is non-decreasing  so if $1\le k\le x<k+1$, then $x/k\le2$ and we have
\[
\frac12\int_0^xu\le\int_0^ku\le\int_0^xu\quad\text{and}\quad
\int_0^{1/x}w^{1-p'}\le\int_0^{1/k}w^{1-p'}\le2\int_0^{1/x}w^{1-p'}.
\]
Thus, for $k=1,2,\dots$,
\[
\int_k^{k+1}\bigg(\int_0^xu\bigg)^{r/p}\bigg(\int_0^{1/x} w^{1-p'}\bigg)^{r/p'}u(x)\,dx
\sim\bigg(\int_0^ku\bigg)^{r/p}\bigg(\int_0^{1/k} w^{1-p'}\bigg)^{r/p'}u(k).
\]
The notation $A\sim B$, in this case, means that $2^{-r/p}A\le B\le 2^{r/p'}A$.

Since $\int_0^ku=\sum_{j=0}^{k-1}u(j)$ we can sum over $k$ to see that (\ref{MH}) is equivalent to (\ref{FS}).

A similar analysis shows that the sufficient condition (\ref{BH0cond}) is equivalent to
\[
\sup_{k=1,2,\dots}\bigg(\sum_{j=0}^{k-1}u(j)\bigg)^{1/q}\bigg(\int_0^{1/k} w^{1-p'}\bigg)^{1/p'}<\oo.
\]
Corresponding reductions may be carried out for all the weight conditions encountered in Theorem \ref{Summary}.

\section{Examples}\label{counterexamples}

In this section we produce explicit weights $U$ and $W$, depending on indices $p$ and $q$ with $1<q<2<p<\oo$, for which the Hardy inequality (\ref{Main Hardy}) holds but the Fourier inequality (\ref{Lebesgue}) fails. For simplicity only the case $n=1$ is considered.

Note that, by Proposition \ref{Main Hardy Weights}, proving (\ref{MH}) is enough to show that (\ref{Main Hardy}) holds.

Example \ref{Ex1}, suggested by \cite{G}*{Exercise 3.1.6}, looks at the case of the Fourier transform on $\mathbb T$, where the compactness of $\mathbb T$ permits a straightforward argument. Example \ref{Ex2} uses the same basic approach but is technically more complicated. A Gaussian function is introduced to ensure convergence in the absence of compactness in the domain space.

\begin{ex}\label{Ex1} Suppose $p$, $q$ and $r$ satisfy $1<q<2<p<\oo$ and $1/r=1/q-1/p$. Take $\alpha,\beta\in (0,1)$ so that
\[
\frac12<\frac\beta q<\frac\alpha{p'}<\frac1q<1.
\]
Let $U(k)=(|k|+1)^{\beta-1}$, $k\in\mathbb Z$, and $W(x)=x^{(1-\alpha)(p-1)}$, $0\le x<1$. Let $u$ be the rearrangement of $U$ with respect to counting measure on $\mathbb Z$ and define $w$ by $1/w=(1/W)^*$ with respect to Lebesgue measure on $\mathbb T=[0,1]$. Then condition (\ref{MH}) holds but the Fourier series inequality
\[ 
\bigg(\sum_{k\in \mathbb Z} |\hat f(k)|^qU(k)\bigg)^{1/q}
\le C\bigg(\int_0^1|f(x)|^pW(x)\,dx\bigg)^{1/p}, \quad f\in L^1(\mathbb T),
\]
fails to hold for any $C$.
\end{ex}
\begin{proof}
Observe that $w(t)=t^{(1-\alpha)(p-1)}$ for $0\le t<1$ and $w(t)=\oo$ for $t>1$. Also, $u(t)=1$ for $0\le t<1$, and $u(t)=(k+1)^{\beta-1}$ for $2k-1\le t<2k+1$, $k=1,2,\dots$. It follows that $u(t)\le((t+1)/2)^{\beta-1}$ and easy estimates show that (\ref{MH}) holds.

On the other hand, completing \cite{G}*{Exercise 3.1.6} shows that the sum
\[
g(x)=\sum_{k=2}^\oo k^{-1/2}(\log k)^{-2}e^{ik\log k}e^{2\pi ikx}
\]
defines a continuous function on $\mathbb T$. In particular, $g$ is bounded on $[0,1]$ and thus $\|g\|_{L^p(W)}<\oo$ because $W$ is integrable. However, $\hat g(k)=k^{-1/2}(\log k)^{-2}e^{ik\log k}$ for $k=2,3,\dots$, so
\[
\|\hat g\|_{\ell^q(U)}^q\ge\sum_{k=2}^\oo k^{-q/2}(\log k)^{-2q}(|k|+1)^{\beta-1}=\oo.
\]
Thus, the Fourier series inequality fails with $f=g$.
\end{proof}
Before beginning with our second example, we set up to use van der Corput's lemma several times.
\begin{lem}\label{vdC} Suppose $1<a<b$, $x\in \mathbb R$, and $f\ge0$ is continuously differentiable. If $f$ is decreasing on $[a,b]$ then
\[
\bigg|\int_a^b f(y)e^{iy\log(y)}e^{2\pi ixy}\,dy\bigg|\le24a^{1/2}f(a)+12\int_a^by^{-1/2}f(y)\,dy.
\]
If $f$ is increasing on $[a,b]$ then
\[
\bigg|\int_a^b f(y)e^{iy\log(y)}e^{2\pi ixy}\,dy\bigg|\le24b^{1/2}f(b).
\]
\end{lem}
\begin{proof} Suppose $[\bar a,\bar b]\subseteq[a,b]$. Letting $y\mapsto \bar by$, we have
\[
\int_{\bar a}^{\bar b} e^{iy\log(y)}e^{2\pi ixy}\,dy
=\bar b\int_{\bar a/\bar b}^1 e^{i\bar by(\log(y)+\log(\bar b)+2\pi x)}\,dy.
\]
The second derivative of $y(\log(y)+\log(\bar b)+2\pi x)$ is $1/y$, which is at least 1 when $0<y\le1$. So the van der Corput lemma \cite{G}*{Proposition 2.6.7(b)} (with $k=2$) shows that
\[
\bigg|\int_{\bar a}^{\bar b} E\,dy\bigg|\le \bar b(24\bar b^{-1/2})=24\bar b^{1/2},
\]
where we have written $E=e^{iy\log(y)}e^{2\pi ixy}$ for simplicity. 

First suppose that $f$ is decreasing. Then $-f'(z)=|f'(z)|$ for $a<z<b$ so 
\begin{align*}
\bigg|\int_a^b f(y)E\,dy\bigg|&=\bigg|\int_a^b \bigg(f(b)+\int_y^b |f'(z)|\,dz\bigg)E\,dy\bigg|\\
&\le f(b)\bigg|\int_a^b E\,dy\bigg|+\int_a^b \bigg|\int_a^z E\,dy\bigg| |f'(z)|\,dz\\
&\le 24b^{1/2}f(b)+24\int_a^b z^{1/2}|f'(z)|\,dz\\
&= 24b^{1/2}f(b)+24\bigg( a^{1/2}f(a)-b^{1/2}f(b)+\frac12\int_a^b y^{-1/2}f(y)\,dy\bigg)\\
&= 24a^{1/2}f(a)+12\int_a^b y^{-1/2}f(y)\,dy.
\end{align*}
Next suppose $f$ is increasing. Then $f'(z)\ge0$ for $a<z<b$ so 
\begin{align*}
\bigg|\int_a^b f(y)E\,dy\bigg|&=\bigg|\int_a^b \bigg(f(a)+\int_a^y f'(z)\,dz\bigg)E\,dy\bigg|\\
&\le f(a)\bigg|\int_a^b E\,dy\bigg|+\int_a^b \bigg|\int_z^b E\,dy\bigg| f'(z)\,dz\\
&\le 24b^{1/2}f(a)+24\int_a^b b^{1/2}f'(z)\,dz\\
&= 24b^{1/2}f(b).\qedhere
\end{align*}
\end{proof}
\begin{ex}\label{Ex2} Suppose $p$, $q$ and $r$ satisfy $1<q<2<p<\oo$ and $1/r=1/q-1/p$. Take $\alpha,\beta\in (0,1)$ so that
\[
\frac12<\frac\beta q<\frac\alpha{p'}<\frac1q<1.
\]
Let $U(y)=(|y|+1)^{\beta-1}$ for $y\in \mathbb R$ and $W(x)=|x|^{(1-\alpha)(p-1)}$ for $x\in \mathbb R$. Define $u$ and $w$ by $u=U^*$ and $1/w=(1/W)^*$. Then condition (\ref{MH}) holds but the Fourier inequality (\ref{Lebesgue}) fails to hold for any constant $C$.
\end{ex}
\begin{proof}
Observe that $w(t)=(t/2)^{(1-\alpha)(p-1)}$ and $u(t)=((t/2)+1)^{\beta-1}$. With these in hand, easy estimates show that (\ref{MH}) holds.

For each $K>2+e$, let $\gamma=(2/3)(1+\log K)^{3/2}$ and define $m,n\in (0,1)$ by requiring that $(1+\log K)m=\frac\pi3$ and $(1+\log K)n=\frac\pi2$. Take
\[
h(x)=\frac{\sqrt{2\pi}}\gamma e^{-2\pi^2(x/\gamma)^2}\quad\text{to get}\quad
\hat h(y)=e^{-\frac12y^2\gamma^2}.
\]
Also, take
\[
g(x)=h(x)\int_e^Ky^{-1/2}\log(y)^{-2}e^{iy\log(y)}e^{2\pi ixy}\,dy.
\]
Clearly $g\in L^1$ and 
\begin{align*}
\hat g(z)=\int_{-\oo}^\oo g(x)e^{-2\pi ixz}\,dx
&=\int_e^Ky^{-1/2}\log(y)^{-2}e^{iy\log(y)}\hat h(y-z)\,dy\\
&=\int_e^Ky^{-1/2}\log(y)^{-2}e^{iy\log(y)}e^{-\frac12(y-z)^2\gamma^2}\,dy.
\end{align*}

First we estimate $\|g\|_{L^p(W)}$. Since $y^{-1/2}\log(y)^{-2}$ is decreasing for $y\ge e$, Lemma \ref{vdC} gives,
\[
\bigg|\int_e^Ky^{-1/2}\log(y)^{-2}e^{iy\log(y)}e^{2\pi ixy}\,dy\bigg|
\le24+12\int_e^K\log(y)^{-2}\,\frac{dy}y\le 36.
\]
Using this estimate, and making the substitution $x\mapsto\gamma x$, yields
\[
\|g\|_{L^p(W)}\le 36\|h\|_{L^p(W)}=\gamma^{-\alpha/p'}36\sqrt{2\pi}\bigg(\int_{-\oo}^\oo e^{-2\pi^2px^2}|x|^{(1-\alpha)(p-1)}\,dx\bigg)^{1/p}.
\]
Since $\alpha/p'>0$ and $(1-\alpha)(p-1)>-1$, this norm goes to zero as $K\to\oo$.

Estimating $\|\hat g\|_{L^q(U)}$ is next. Suppose $e+1<z<K-1$. Then,
\begin{align*}
|\hat g(z)|
&\ge \bigg|\int_{z-n}^{z+n}y^{-1/2}\log(y)^{-2}e^{-\frac12(y-z)^2\gamma^2}e^{iy\log(y)}\,dy\bigg|\\
&-\bigg|\int_e^{z-n}y^{-1/2}\log(y)^{-2}e^{-\frac12(y-z)^2\gamma^2}e^{iy\log(y)}\,dy\bigg|\\
&-\bigg|\int_{z+n}^Ky^{-1/2}\log(y)^{-2}e^{-\frac12(y-z)^2\gamma^2}e^{iy\log(y)}\,dy\bigg|\\
&\equiv A-B_1-B_2.
\end{align*}

To estimate $A$, we will multiply by $1=|e^{-iz\log(z)}|$ and then reduce the modulus to its real part.
This is justified by the mean value theorem: If $z-n<y<z+n$ then for some $\bar y\in (z-n,z+n)\subseteq(e,K)$, 
\[
|y\log(y)-z\log(z)|=(1+\log \bar y)|y-z|\le(1+\log K)n=\pi/2
\]
so $\cos(y\log(y)-z\log(z))\ge0$. Similarly, if $z-m<y<z+m$ then $|y\log(y)-z\log(z)|\le\pi/3$ so $\cos(y\log(y)-z\log(z))\ge1/2$. We have, 
\begin{align*}
A&=\bigg|\int_{z-n}^{z+n}y^{-1/2}\log(y)^{-2}e^{-\frac12(y-z)^2\gamma^2}e^{iy\log(y)-iz\log(z)}\,dy\bigg|\\
&\ge\int_{z-n}^{z+n}y^{-1/2}\log(y)^{-2}e^{-\frac12(y-z)^2\gamma^2}\cos(y\log(y)-z\log(z))\,dy
\\
&\ge\frac12\int_{z-m}^zy^{-1/2}\log(y)^{-2}e^{-\frac12(y-z)^2\gamma^2}\,dy\\
&\ge\frac12z^{-1/2}(\log z)^{-2}\int_{z-m}^ze^{-\frac12(y-z)^2\gamma^2}\,dy\\
&=\frac12z^{-1/2}(\log z)^{-2}\gamma^{-1}\int_0^{m\gamma}e^{-\frac12y^2}\,dy\\
&\ge\frac14z^{-1/2}(\log z)^{-2}\gamma^{-1}.
\end{align*}
The last estimate uses the fact that $m\gamma=(\pi/3)(2/3)(1+\log K)^{1/2}>1$ to show that 
\[
\int_0^{m\gamma}e^{-\frac12y^2}\,dy\ge\int_0^1e^{-\frac12}\,dy\ge\frac12.
\]

To estimate $B_1$ we apply Lemma \ref{vdC} with $x=0$.  Since $n\gamma^2=(\pi/2)(2/3)^2(1+\log K)^2\ge5/(2e)$ we can show that $y^{-1/2}\log(y)^{-2}e^{-\frac12(y-z)^2\gamma^2}$ is an increasing function of $y$ on $[e,z-n]$ by estimating the derivative of its logarithm to get
\[
-\frac1{2y}-\frac2{y\log(y)}+(z-y)\gamma^2\ge-\frac1{2e}-\frac2e+n\gamma^2\ge0.
\]
Thus,
\[
B_1\le24 \log(z-n)^{-2}e^{-\frac12n^2\gamma^2}\le24e^{-\frac12n^2\gamma^2}
=24e^{-\frac12(\frac\pi2)^2(\frac23)^2(1+\log K)}\le24K^{-\pi^2/18}.
\]

For $y>z$ it is clear that the function $y^{-1/2}\log(y)^{-2}e^{-\frac12(y-z)^2\gamma^2}$ is decreasing, so we may use Lemma \ref{vdC}, with $x=0$, to estimate $B_2$ as well. We have
\begin{align*}
B_2&\le24\log(z+n)^{-2}e^{-\frac12n^2\gamma^2}+12\int_{z+n}^K\log(y)^{-2}e^{-\frac12(y-z)^2\gamma^2}\,\frac{dy}y\\
&\le 24\log(z+n)^{-2}e^{-\frac12n^2\gamma^2}
+12e^{-\frac12n^2\gamma^2}\log(z+n)^{-1}\le36K^{-\pi^2/18}.
\end{align*}

Since $\pi^2/18>1/2$, 
\[
\frac12A\ge \frac18K^{-1/2}(\log K)^{-2}\frac32(1+\log K)^{-3/2}\ge 60 K^{-\pi^2/18}=B_1+B_2.
\]
for $K$ sufficiently large. It follows that for such $K$, and $z\in (e+1,K-1)$,
\[
|\hat g(z)|\ge A-B_1-B_2\ge\frac12A\ge\frac18z^{-1/2}(\log z)^{-2}\gamma^{-1}.
\]
Now
\begin{align*}
\|\hat g\|_{L^q(U)}&\ge \frac18\gamma^{-1}\bigg(\int_{e+1}^{K-1} 
z^{-q/2}(\log z)^{-2q}(z+1)^{\beta-1}\,dz\bigg)^{1/q}\\
&\ge \frac3{16}(1+\log K)^{-3/2}(1+\log K)^{-2}\bigg(\int_{e+1}^{K-1} 
(z+1)^{-q/2+\beta-1}\,dz\bigg)^{1/q}\\
&= \frac3{16}(1+\log K)^{-7/2}\Big(\beta-\frac q2\Big)^{-1/q}
(K^{\beta-q/2}-(e+2)^{\beta-q/2})^{1/q}.
\end{align*}
This goes to infinity as $K\to\oo$. We conclude that there is no finite constant $C$ such that $\|\hat f\|_{L^q(U)}\le C\|f\|_{L^p(W)}$ for all $f\in L^1$.
\end{proof}

As we have pointed out earlier, the weight conditions (\ref{MH}) and (\ref{alt}) are equivalent when $q>1$. Thus the weight condition of \cite{BH2003}*{Theorem 1(ii)} holds in the example above, yet the Fourier inequality fails. 

The weight condition of \cite{BH2003}*{Theorem 4(iv)} holds as well because, as we have seen in Section \ref{Comp14}, the weight conditions of \cite{BH2003}*{Theorems 1 and 4} are equivalent whenever the appropriate a priori condition holds. To see that the weight $u$ of Example \ref{Ex2} satisfies $u^{1-q'}\in B_{q'}$, observe that the continuous function of $x$ given by
\[
\frac{x^{q'}\int_x^\oo \big(\frac t2+1\big)^{(1-\beta)(q'-1)}t^{-q'}\,dt}{\int_0^x\big(\frac t2+1\big)^{(1-\beta)(q'-1)}\,dt}
\]
is bounded as $x\to0$ and as $x\to\oo$. The details are omitted.

\end{document}